\documentclass[12pt]{amsart}
\usepackage{txfonts}
\usepackage{amscd,amssymb,mathabx,verbatim}
\usepackage[all]{xy}
\usepackage{graphicx,calrsfs}
\usepackage[colorlinks,plainpages,backref,urlcolor=blue]{hyperref}
\usepackage{tikz}
\usetikzlibrary{calc}
\usepackage{mdwlist}

\newtheorem{theorem}{Theorem}[section]
\newtheorem{corollary}[theorem]{Corollary}
\newtheorem{lemma}[theorem]{Lemma}
\newtheorem{prop}[theorem]{Proposition}

\theoremstyle{definition}
\newtheorem{definition}[theorem]{Definition}
\newtheorem{example}[theorem]{Example}
\newtheorem{remark}[theorem]{Remark}
\newtheorem{conjecture}[theorem]{Conjecture}
\newtheorem{question}[theorem]{Question}

\newtheorem*{ack}{Acknowledgments}

\newcommand{\Z}{\mathbb{Z}}
\newcommand{\Q}{\mathbb{Q}}

\newcommand{\C}{\mathbb{C}}
\newcommand{\FF}{\mathbb{F}}
\renewcommand{\L}{\mathbb{L}}

\renewcommand{\k}{\Bbbk}

\DeclareMathAlphabet{\pazocal}{OMS}{zplm}{m}{n}
\newcommand{\A}{{\pazocal{A}}}

\newcommand{\HH}{{\pazocal H}}

\newcommand{\RR}{{\mathcal R}}
\newcommand{\VV}{{\mathcal V}}

\newcommand{\F}{{\mathcal{F}}}
\newcommand{\cP}{{\mathcal{P}}}
\newcommand{\cE}{{\mathcal{E}}}
\newcommand{\M}{{\mathcal{M}}}
\newcommand{\cC}{{\mathcal{C}}}
\newcommand{\wC}{\,\widehat{\mathcal{\!C}}}

\newcommand{\g}{{\mathfrak{g}}}
\newcommand{\h}{{\mathfrak{h}}}

\newcommand{\m}{{\mathfrak{m}}}

\newcommand{\rr}{{\mathfrak{r}}}

\DeclareMathOperator{\gr}{gr}
\DeclareMathOperator{\im}{im}

\DeclareMathOperator{\ab}{{ab}}

\DeclareMathOperator{\Hom}{{Hom}}

\DeclareMathOperator{\ad}{ad}

\DeclareMathOperator{\Lie}{Lie}

\DeclareMathOperator{\lcs}{LCS}
\newcommand{\oX}{\overline{X}}

\newcommand{\surj}{\twoheadrightarrow}
\newcommand{\inj}{\hookrightarrow}

\newcommand{\isom}{\xrightarrow{\,\cong\,}}
\newcommand{\abs}[1]{\left| #1 \right|}

\def\dot{\mathchar"013A}  
\newcommand{\hdot}{{\raise1pt\hbox to0.35em{\Huge $\dot$}}} 
\newcommand{\bwedge}{\mbox{\normalsize $\bigwedge$}}

\newcommand{\what}{\widehat{\:\:}}

\newenvironment{romenum}
{ 

\begin{enumerate}}{\end{enumerate}}

\newcommand{\cga}{\ensuremath{\small{\textsf{cga}}}}
\newcommand{\cdga}{\ensuremath{\small{\textsf{cdga}}}}

\begin{document}

\title[Infinitesimal finiteness obstructions]{%
Infinitesimal finiteness obstructions}

\author[Stefan Papadima]{Stefan Papadima$^1$$\dagger$}
\address{Simion Stoilow Institute of Mathematics, 
P.O. Box 1-764,
RO-014700 Bucharest, Romania}
\email{\href{mailto:Stefan.Papadima@imar.ro}{Stefan.Papadima@imar.ro}}
\thanks{$^1$Work partially supported by the Romanian Ministry of Research and
Innovation, CNCS-UEFISCDI, grant
PN-III-P4-ID-PCE-2016-0030, within PNCDI III}
\thanks{$\dagger$Deceased January 10, 2018}

\author[Alexander~I.~Suciu]{Alexander~I.~Suciu\,$^2$}
\address{Department of Mathematics,
Northeastern University,
Boston, MA 02115, USA}
\email{\href{mailto:a.suciu@northeastern.edu}{a.suciu@northeastern.edu}}
\urladdr{\href{http://web.northeastern.edu/suciu/}%
{web.northeastern.edu/suciu/}}
\thanks{$^2$Work partially supported by the Simons Foundation Collaboration 
Grant for Mathematicians \#354156}



\begin{abstract}
Does a space enjoying good finiteness properties admit an 
algebraic model with commensurable finiteness properties?  
In this note, we provide a rational homotopy obstruction for 
this to happen.  As an application, we show that the maximal 
metabelian quotient of a very large, finitely generated group is not 
finitely presented. Using the theory of $1$-minimal models, we 
also show that a finitely generated group 
$\pi$ admits a connected $1$-model with finite-dimensional 
degree $1$ piece if and only if the Malcev Lie algebra $\m(\pi)$ 
is the lower central series completion of a finitely presented Lie algebra. 
\end{abstract}

\maketitle
\setcounter{tocdepth}{1}
\tableofcontents

\section{Introduction and statement of results}
\label{sect:intro}

\subsection{Finite $\cdga$ models}
\label{subsec:intro1}

A recurring theme in topology is to determine the geometric and homological 
finiteness properties of spaces and groups. A prototypical such question is 
to determine whether a path-connected space $X$ is homotopy equivalent 
to a CW-complex with finite $q$-skeleton, for some $1\le q \le \infty$, 
in which case we say $X$ is {\em $q$-finite}.  Another question  
is to decide whether a finitely generated group $\pi$ admits a finite 
presentation, or, more generally, a classifying space $K(\pi,1)$ with 
finite $q$-skeleton. 

A fruitful approach to this type of question is to compare 
the finiteness properties of the spaces or groups under 
consideration to the corresponding finiteness properties 
of algebraic models for such spaces and groups. To formulate 
our motivating question, we need some terminology.

Let $A$ be commutative differential graded algebra 
(for short, a $\cdga$) over a field $\k$ of characteristic $0$.  
By analogy with the aforementioned 
topological notion, we say that $A$ is 
{\em $q$-finite}\/ if it is connected (i.e., $A^0= \k \cdot 1$) 
and $\sum_{i\le q} \dim A^i < \infty$. 
Furthermore, we say that two $\cdga$s $A$ and $B$ 
have the same (homotopy) $q$-type 
(written $A\simeq_q B$) if there is a zig-zag of $\cdga$ 
maps  connecting $A$ and $B$, with each such map 
inducing isomorphisms in homology up to
degree $q$ and a monomorphism in degree $q+1$. 

We say that a $\cdga$ $A$ is a 
{\em $q$-model}\/ for a space $X$ if it has the same $q$-type as 
Sullivan's $\cdga$ of piecewise polynomial, complex-valued 
forms on $X$.  The basic question that we shall address in this 
paper  is the following.

\begin{question}
\label{mainpbm}
When does a $q$-finite space $X$ admit a $q$-finite $q$-model $A$?
\end{question}

An important motivation for this question comes from work 
of Dimca--Papadima \cite{DP-ccm} and Budur--Wang \cite{BW}, 
who discovered some deep connections between the finiteness properties 
of algebraic models for spaces and the structure of the corresponding 
cohomology jump loci.  

Observe that $X$ is $1$-finite if and only if the group $\pi=\pi_1(X)$ 
is finitely generated; moreover, if $X$ is $2$-finite, then $\pi$ is finitely 
presented.  Further motivation for considering Question \ref{mainpbm} 
comes from an effort to understand whether the maximal metabelian 
quotient $\pi/\pi''$ of a finitely presented group $\pi$ is also finitely presentable.

\subsection{Finiteness obstructions}
\label{subsec:intro2}

Our first result (which will be proved in \S\ref{subsec:infobs}), 
provides an infinitesimal obstruction to a positive answer to 
the above question. 

\begin{theorem}
\label{thm:betti}
Let $X$ be a space which admits a $q$-finite $q$-model. If $\M_q (X)$ is the 
Sullivan $q$-minimal model $\cdga$ of $X$, then $\dim H^i(\M_q(X))< \infty$, 
for all $i\le q+1$.
\end{theorem}

In \cite{BW17}, Budur and Wang  found a completely different 
finiteness obstruction, involving the structure of the cohomology 
jump loci for  rank $1$ local systems on $X$. Namely, if $X$ 
is as above and $q$-finite, then all the irreducible components passing through 
the origin of those jump loci (in degree at most $q$) are algebraic 
subtori of the character group of $\pi_1(X)$.   As shown in 
Example \ref{ex:subtler}, our infinitesimal obstruction may be 
subtler than this jump loci test.

As an application of Theorem \ref{thm:betti}, we produce a large class of 
finitely generated groups whose maximal metabelian quotients 
have no good finiteness properties, either at the level of 
presentation complexes, or at the level of $1$-models. 
First, some quick definitions. A group $G$ is said to be 
{\em very large}\/ if it has a free, non-cyclic quotient; the 
group $G$ is merely {\em large}\/ if it has a finite-index 
subgroup which is very large. 
Our result (which will be proved in \S\ref{subsec:pfmeta}), 
may be stated as follows.

\begin{theorem}
\label{thm:meta}
Let $G$ be a metabelian group of the form $G=\pi/\pi''$, where  
$\pi$ is a finitely generated, very large group, and $\pi''$ is its 
second derived subgroup.  Then: 
\begin{enumerate}
\item \label{met1} 
$G$ is not finitely presentable.
\item \label{met2} 
$G$ does not admit a $1$-finite $1$-model.
\end{enumerate}
\end{theorem}

Consequently, if $\pi$ is a group as in Theorem \ref{thm:meta}, 
then the derived subgroup $G'$ is not finitely generated. We 
observe in Example \ref{ex:largefp} that the condition that the 
group $\pi$ be very large cannot be relaxed to it only being large.

\subsection{Malcev and holonomy Lie algebras}
\label{subsec:intro3}

In the last part of this paper, we turn to studying the finiteness properties 
of Lie algebras. 
We start in \S\ref{sec:holmal} by analyzing the holonomy Lie algebra, 
$\h(A)$, associated to a $1$-finite $\cdga$ $A$. 
Using the Chevalley--Eilenberg cochain functor $\cC$, we
define a {\em functorial} $1$-minimal $\cdga$, $\wC (\h(A))$, and a 
natural $\cdga$ transformation, $f_{A} \colon \wC (\h(A)) \rightarrow A$.
Our main technical result (which we prove in \S\ref{subsec:pfclass}), 
reads as follows.

\begin{theorem}
\label{thm:nat1model-intro}
The classifying map $f_{A}$ is a functorial $1$-minimal model map 
for $A$.
\end{theorem}

As an application, we prove  in Corollary \S\ref{cor:malholo} the following:
If a finitely generated group $\pi$ admits a $1$-finite $1$-model $A$, then 
the Malcev Lie algebra $\m (\pi)$ is isomorphic to the lower central series (LCS) 
completion of the holonomy Lie algebra $\h (A)$.  This recovers a result from \cite{BMPP}, 
which in turn generalizes a result  from \cite{Bez}.  

Finally, we use Theorem \ref{thm:nat1model-intro} to prove 
in \S\ref{sec:complobstr} the following theorem, which 
provides a complete answer to Question \ref{mainpbm} 
in the case when $q=1$. 

\begin{theorem}
\label{thm:complobstr-intro}
A space with finitely generated fundamental group $\pi$ admits a 
$1$-finite $1$-model if and only if the Malcev Lie algebra $\m(\pi)$
is the LCS completion of a finitely presented Lie algebra.
\end{theorem}

\section{Algebraic models and finiteness obstructions}
\label{sect:infobs}

\subsection{Differential graded algebras}
\label{subsec:cdga}
We shall fix throughout a ground field $\k$ of characteristic $0$. 
A commutative, differential graded algebra (\cdga) over $\k$ 
is a positively-graded $\k$-vector space, $A=\bigoplus_{i\ge 0} A^i$, 
endowed with a graded-commutative multiplication 
map $\cdot\colon A^i \otimes A^j \to A^{i+j}$, and a differential 
$d\colon A^i\to A^{i+1}$ satisfying $d(a\cdot b) = da\cdot b 
+(-1)^{i} a \cdot db$, for every $a\in A^i$ and $b\in A^j$. 
We let $A=(A^{\hdot},d)$ denote such an object, and write  
$Z^i(A)= \ker (d\colon A^i \to A^{i+1})$, $B^i(A)= \im (d\colon A^{i-1} \to A^i)$, 
and $H^i(A)=Z^i(A)/B^i(A)$. The cohomology of the underlying cochain complex, 
$H^{\hdot}(A)$, is a commutative, graded algebra ($\cga$); 
we let $b_i(A)=\dim H^i(A)$ be its Betti numbers.

Fix an integer $q\ge 1$.  (We will also allow $q=\infty$, although 
we will usually omit $q$ from terminology and notation in that case.) 
We say that two $\cdga$s $A$ and $B$ have {\em the same (homotopy) 
$q$-type}\/ if there is a zig-zag of $\cdga$ maps from $A$ to $B$, 
\begin{equation}
\label{eq:ziggy}
\xymatrix{
A  & A_1 \ar_(.45){\varphi_1}[l]  \ar^{\varphi_2}[r] & \cdots 
& A_{k-1}   \ar[l]\ar^{\varphi_{k}}[r] & B},
\end{equation}
with each map $\varphi_j$ being a {\em $q$-equivalence}, i.e., inducing 
isomorphisms in homology up to degree $q$ and a monomorphism in degree $q+1$.  
(If $q=\infty$, the maps $\varphi_j$ are also called quasi-isomorphisms.) 
Clearly, if $A\simeq_q B$, then $b_i(A)=b_i(B)$ for all $i\le q$.  

Let $\bigwedge V$ be the free graded-commutative 
algebra generated by the graded vector space $V=\bigoplus_{i>0}V^i$. 
Following \cite{Su77,Mo}, we say that a $\cdga$ is {\em $q$-minimal}\/ 
if it is of the form $(\bigwedge V, d)$, where the differential structure is 
the inductive limit of a sequence of Hirsch extensions of increasing 
degrees, and $V^i=0$ for $i>q$, if $q\ne \infty$. 

A {\em $q$-minimal model map}\/ for a $\cdga$ $A$ is a $q$-equivalence 
$(\bigwedge V, d)\to A$ with $(\bigwedge V, d)$ $q$-minimal.
Every $\cdga$ $A$ with connected homology admits such a map. Moreover, 
the isomorphism type of the $\cdga$ $(\bigwedge V, d)$ is uniquely determined;  
this $\cdga$ is called {\em the $q$-minimal model of $A$}, and is 
denoted by $\M_q (A)$. It is readily seen that two $\cdga$s with connected 
cohomology have the same $q$-type if and only if their $q$-minimal 
models are isomorphic. 

\subsection{Algebraic models for spaces and groups}
\label{subsec:algmod}

Given a topological space $X$, we let $\Omega^{\hdot}(X)$ 
be Sullivan's algebra \cite{Su77}  of piecewise polynomial, 
complex-valued forms on $X$, see also \cite{FOT,FHT}. 
This is a functorially defined $\cdga$ with the property that 
$H^{\hdot}(\Omega(X))\cong H^{\hdot}(X,\C)$, as graded rings.  
When $X$ is a smooth manifold, $\Omega(X) \simeq \Omega_{{\rm dR}}(X)$,  
de Rham's algebra of smooth $\C$-forms on $X$.  Furthermore,  
if $X$ is a simplicial complex, then 
$\Omega(X)\simeq \Omega_{{\rm s}}(X)$, the algebra 
of piecewise polynomial, $\C$-forms on the simplices of $X$. 

We say that a $\cdga$ $(A,d)$ is a $q$-model for $X$ 
if $A\simeq_q \Omega(X)$. By considering the classifying space 
$X=K(\pi, 1)$ of a group $\pi$, and replacing $X$ by $\pi$ in both 
terminology and notation, we may speak about $q$-minimal models,
$q$-types and finiteness properties of groups, in the sense from 
Question \ref{mainpbm}.   If $A$ is $q$-finite and $A\simeq_q \Omega(X)$, 
then clearly $b_i(X)=b_i(A)<\infty$, for all $i\le q$ (or all $i$, if $q=\infty$).  
Moreover, if $\pi=\pi_1(X)$ is the fundamental group of a path-connected 
space $X$, then any classifying map $X\to K(\pi,1)$ induces an isomorphism 
between the corresponding $1$-minimal models, $\M_1(X)\cong \M_1(\pi)$.

A continuous map $f\colon X\to Y$ is said to be a {\em $q$-rational homotopy 
equivalence}\/ if the induced map $f^*\colon H^{\hdot}(Y,\Q)\to H^{\hdot}(X,\Q)$ 
is an isomorphism in degrees up to $q$ and a monomorphism in degree $q+1$.  
Clearly, such a map induces a $q$-equivalence 
$\Omega(Y)\simeq_q \Omega(X)$. Consequently, the  
existence of a $q$-finite $q$-model for a space $X$ is an invariant of 
$q$-rational homotopy type, and thus, of $q$-homotopy type. In particular, 
the existence a $1$-finite $1$-model for a path-connected  space $X$ 
is equivalent to the existence a $1$-finite $1$-model for its fundamental 
group $\pi=\pi_1(X)$.

Unless otherwise specified, all spaces we consider here will 
be path-connected, and will have the homotopy type of a 
CW-complex; for short, we will call such objects {\em CW-spaces}. 
Oftentimes, the geometry of a space  forces the existence of 
a finite model for it.  Examples of spaces having finite models 
include quasi-projective manifolds, compact solvmanifolds,
K\"{a}hler manifolds, Sasakian manifolds, and principal bundles 
with compact, connected structural group over finite CW-complexes 
having finite models; see for instance \cite{DP-ccm,FOT,PS17} 
and references therein.

\subsection{Nilpotent spaces and formal spaces}
\label{subsec:nilpformal}

A CW-space $X$ is said to be {\em nilpotent}\/ 
if the fundamental group $\pi=\pi_1(X)$ is nilpotent 
and acts unipotently on $\pi_n (X)$ for all $n>1$.
Nilpotent spaces provide a class of examples for which the answer 
to Question \ref{mainpbm} is unobstructed.

\begin{theorem}[\cite{Su77}]
\label{thm:nilpsp}
Let $X$ be a nilpotent CW-space. 
\begin{enumerate}
\item \label{s1} 
If all the Betti numbers of $X$ are finite, 
then $X$ admits a $q$-finite $q$-model, for all $1\le q<\infty$.

\item \label{s2}
Moreover, if $\dim H_{\hdot}(X, \k)<\infty$, then $X$ admits a finite model.
\end{enumerate}
\end{theorem}

\begin{proof}
Sullivan proved in \cite{Su77} that the minimal model of a nilpotent 
CW-space with finite Betti numbers is of the form $\M (X)= (\bigwedge V, d)$,
where $V$ is a graded vector space of finite type. Hence, $\M (X)$ is of finite 
type, as a graded vector space, from which claim \ref{s1}  follows. 

Assume now that $H^{>n}(X)=0$, for some $n>0$. Pick a vector 
space decomposition, $\M^n (X)= Z^{n}(\M (X)) \oplus C^{n}$.  Plainly,  
the direct sum $J= \M^{\ge n+1}(X) \oplus C^n$ is an acyclic
differential graded ideal of $\M (X)$. By construction, 
$\Omega (X)\simeq \M (X)/J$, and the $\cdga$ $\M (X)/J$ is finite;  
thus, claim \ref{s2} is also verified.
\end{proof}

A space $X$ is said to be {\em $q$-formal}\/ if 
$\Omega (X)\simeq_q (H^{\hdot}(X,\C), d=0)$. 
For this interesting class of spaces, 
the answer to Question \ref{mainpbm} is again unobstructed: clearly, if 
a space $X$ is $q$-formal and $q$-finite, then it has the $q$-finite $q$-model 
$(H^{\hdot}(X,\C), d=0)$. 

\subsection{On the Betti numbers of minimal models}
\label{subsec:infobs}

We now go beyond nilpotent rational homotopy theory 
and formal spaces. We fix $1\le q< \infty$.

\begin{lemma}
\label{lem:morefin}
Given a $q$-finite $\cdga$ $A$, there is a natural equivalence 
$A \simeq_q A[q]$, where $A[q]$ is a finite $\cdga$, and $A[q]^{>q+1}=0$.
\end{lemma}

\begin{proof}
The construction is in two steps. First, we replace $A$ by the $\cdga$ 
$A/A^{>q+1}$. Since clearly the natural projection,
$A \surj A/A^{>q+1}$, is a $q$-equivalence, we may suppose in 
the second step that $A^{>q+1}=0$. Next, we put 
\begin{equation}
\label{eq:aq}
A[q]:= \bigoplus_{i\le q}A^i \oplus \bigg(dA^q + 
\sum_{i,j \le q}^{i+j =q+1} A^i \cdot A^j\bigg).
\end{equation}

Plainly, $A[q]$ is a sub-$\cdga$ of $A$. 
By construction, the natural inclusion, $A[q] \inj A$, is a 
$q$-equivalence, and $A[q]$ is finite. 
\end{proof}

\begin{theorem}
\label{prop:infobsq}
Assume that either $\Omega (X)\simeq_q A$ with $A$ $q$-finite, or 
$X$ is $(q+1)$-finite. Then $b_i(\M_q(X))< \infty$, for all $i\le q+1$.
\end{theorem}

\begin{proof}
In the first case, we may assume by Lemma \ref{lem:morefin} that the 
$\cdga$ $A$ is actually finite. Since $\Omega (X)\simeq_q A$, the 
differential graded algebra $\M_q(X)$ is the $q$-minimal model of $A$. 
The claim follows from the discussion in \S\ref{subsec:cdga}.

In the second case, the claim follows directly from the existence 
of a $q$-equivalence $\M_q(X)\to \Omega(X)$. 
\end{proof}

\begin{corollary}
\label{cor:infobs1}
Let $\pi$ be a finitely generated group. Assume that either $\pi$ has 
a $1$-finite $1$-model, or $\pi$ is finitely presented. Then $b_2(\M_1(\pi))< \infty$.
\end{corollary}

\begin{proof}
Observe that $\pi$ is 
finitely generated (resp., finitely presented) if and only if $\pi$ admits 
a classifying space  $X=K(\pi,1)$ which is $1$-finite 
(resp., $2$-finite). The claim follows at once from the above theorem, 
by setting $q=1$.
\end{proof}

\subsection{Equivariant algebraic models}
\label{subsec:equiv}

Let $\Phi$ be a finite group acting freely on a space $Y$, 
and let $X=Y/\Phi$ be the orbit space.  Our next goal is 
to compare the algebraic models associated to $Y$ and $X$, 
and understand how the answer to Question \ref{mainpbm} for one 
space affects the answer for the other one.

We start by setting up the category $\Phi$-$\cdga$ (over $\k$): 
the objects are $\cdga$s $A$ endowed with a compatible $\Phi$-action, 
while the morphisms are $\Phi$-equivariant $\cdga$ maps  
$A\to B$. Given a $\Phi$-$\cdga$ $A$, we let $A^{\Phi}$ be the 
sub-$\cdga$ of elements fixed by $\Phi$; there is then a canonical 
$\cdga$ map $A^{\Phi}\to A$. By definition, a $q$-equivalence $A \simeq_q B$ 
in $\Phi$-$\cdga$ ($1\le q\le \infty$) is a zigzag of $\Phi$-equivariant 
$q$-equivalences in $\cdga$.

\begin{lemma}
\label{eq:phicdga}
If $A \simeq_q B$ in $\Phi$-$\cdga$, then $A^{\Phi} \simeq_q B^{\Phi}$ 
in $\cdga$.
\end{lemma}

\begin{proof}
It is readily seen that the fixed points functor $-^{\Phi}$ 
commutes with homology. Thus, this functor takes equivariant $q$-equivalences 
to $q$-equivalences. 
\end{proof}

As is well-known, every CW-complex $X$ has the homotopy type 
of a simplicial complex $K$; moreover, if $X$ has finite $q$-skeleton, 
so does $K$; see \cite[Theorem 2C.5]{Hat}. 
Fix such a triangulation of $X$, and lift it to the cover $Y$.   
The corresponding simplicial Sullivan algebras are 
then related, as follows: 
\begin{equation}
\label{eq:omphi}
\Omega_{{\rm s}}(X)=\Omega_{{\rm s}}(Y)^{\Phi}.
\end{equation}
Using now Lemma \ref{eq:phicdga}, we obtain the following 
result. 

\begin{prop}
\label{prop:equiv}
Let $X$ be a CW-space, and let $Y\to X$ be a finite Galois 
cover, with group of deck transformations $\Phi$.  Let $A$ 
be a $\Phi$-$\cdga$ over $\C$.
\begin{enumerate}
\item \label{phi1}
Suppose $\Omega(Y) \simeq_q A$ in 
$\Phi$-$\cdga$, for some $1\le q\le \infty$. Then 
$\Omega(X) \simeq_q A^{\Phi}$ in $\cdga$. 
\item \label{phi2}
If, moreover, $A$ is $q$-finite, then $A^{\Phi}$ is 
$q$-finite.
\end{enumerate}
\end{prop}

As a consequence, if $Y$ admits an equivariant $q$-finite $q$-model,   
then $X$ admits a $q$-finite $q$-model.  
The $\Phi$-equivariant equivalence hypothesis from 
Proposition \ref{prop:equiv}\ref{phi1} cannot be completely dropped. 
Nevertheless, we venture the following conjecture.

\begin{conjecture}
\label{conj:equiv}
Let $X$ be a connected $CW$-space, 
and let $Y \to X$ be a finite Galois cover with deck group $\Phi$. 
Suppose that $Y$ has finite Betti numbers.
Let $A$ be a $\Phi$-$\cdga$, and assume that there is a zig-zag of 
quasi-isomorphisms connecting $\Omega (Y)$ to $A$ in $\cdga$,
such that the induced isomorphism between $H^{\hdot}(Y,\C)$ 
and $H^{\hdot}(A)$ is $\Phi$-equivariant. Then $A^{\Phi}$ is 
a model for $X$.
\end{conjecture}

In the formal case this conjecture holds, and  
leads to the following result (see also \cite[\S2.2]{BLO}). 

\begin{prop}
\label{prop:equivformal}
Suppose $\Phi$ is a finite group acting simplicially on a formal 
simplicial complex $Y$ with finite Betti numbers.  Then 
the orbit space $X=Y/\Phi$ is again formal.
\end{prop}

\begin{proof}
We deduce from \cite[Corollary 2.9]{Pap} and the remark following it that 
the space $Y$ is formal if and only if $\Omega (Y) \simeq (H^{\hdot}(Y), d=0)$ 
in $\Phi$-$\cdga$. This equivalence 
is based on the Halperin--Stasheff approach to formality from \cite{HSt}. 

Using the natural equivalence $\Omega \simeq \Omega_{{\rm s}}$, we 
now infer from Lemma \ref{eq:phicdga} that 
$\Omega_{{\rm s}} (Y)^{\Phi} \simeq (H^{\hdot}(Y)^{\Phi}, d=0)$ in $\cdga$. 
Observing that equality (\ref{eq:omphi}) holds in this broader setup 
(where the $\Phi$-action is not necessarily free), we conclude that 
$\Omega_{{\rm s}} (X) \simeq (H^{\hdot}(X), d=0)$ in $\cdga$. 
Since $\Omega(X) \simeq \Omega_{{\rm s}}(X)$, we are done.
\end{proof}

\section{Malcev Lie algebras, derived series, and Hall--Reutenauer bases}
\label{sect:meta}

We devote this section to a proof of Theorem \ref{thm:meta}. 

\subsection{Malcev Lie algebras}
\label{subsec:malcev}

We start by recalling some notions from  \cite[Appendix~A]{Qu69}; 
see also \cite{PS04,FHT,SW}. 
A {\em Malcev Lie algebra}\/ is a Lie algebra $\m$ over a field $\k$ 
of characteristic $0$,
endowed with a decreasing, complete $\k$-vector space 
filtration $F=\{F_i\}_{i\ge 1}$ such that $F_1=\m$ and   
$[F_i,F_j]\subset F_{i+j}$, for all $i, j$, and with the property that the 
associated graded Lie algebra, $\gr(\m)=\bigoplus _{i\ge 1}
F_i/F_{i+1}$, is generated in degree~$1$.

For example, the completion $\widehat{L}$ of a Lie algebra $L$ with 
respect to the lower central series ($\lcs$) filtration 
$\Gamma=\{\Gamma_i (L)\}_{i\ge 1}$,  
endowed with the canonical completion filtration, is a 
Malcev Lie algebra. (Note that the $\lcs$ filtration on $L$ 
coincides with the degree filtration, in the case when $L$ is positively graded and
the degree $1$ component $L^1$ generates $L$ as a Lie algebra.)

In \cite{Qu69}, Quillen associates to every group $\pi$, in a functorial 
way, a Malcev Lie algebra, denoted by $\m (\pi)$.  This object, 
called the {\em Malcev completion}\/ of $\pi$, captures 
the properties of the torsion-free nilpotent quotients of $\pi$. 
Here is a concrete way to describe it. 
The group algebra $\k{\pi}$ has a natural Hopf algebra structure,
with comultiplication given by $\Delta(g)=g\otimes g$, and
counit the augmentation map. Let $I$ be the augmentation ideal of $\k{\pi}$. 
One verifies that the Hopf algebra structure on $\k{\pi}$ extends 
to the $I$-adic completion, $\widehat{\k{\pi}}=\varprojlim_r \k{\pi}/I^r$.  
Finally,  $\m (\pi)$ coincides with the Lie algebra of primitive 
elements in $\widehat{\k{\pi}}$, endowed with the inverse limit 
filtration.

For a finitely generated group $\pi$, there is a natural duality between 
$\m (\pi)$ and $\M_1(\pi)$ \cite{Su77}:
the $\cdga$ $\M_1(\pi)$ is the inductive limit $\varinjlim_{i} \cC (\m/F_i)$, 
where $\cC$ is the Chevalley--Eilenberg $\cdga$ cochain functor 
applied to finite-dimensional Lie algebras \cite{CE}.

\subsection{Filtered formality and finiteness properties}
\label{subsec:fffp}
We recall from \cite{SW} that a group $\pi$ is said to be {\em filtered formal}\/ if 
$\m (\pi) \cong \widehat{L}$ as filtered Lie algebras, where $L$ is a graded 
Lie algebra generated in degree $1$, of the form  $L=\L/J$, with  
$\L$ a finitely generated, free, graded Lie algebra and $J\subseteq \L^{\ge 2}$ 
a graded Lie ideal generated in degrees $2$ and higher.

We denote by $(\cdot)^{(k)}$ the $k$-th term of the derived series 
of groups and Lie algebras.   It follows from \cite[Theorem 3.5]{PS04} 
and \cite[Theorem 1.7]{SW} that the groups $\FF_n/\FF_n^{(k)}$ are 
filtered formal, with corresponding Lie algebras $L= \L_n/\L_n^{(k)}$, 
where $\L_n$ denotes the free Lie algebra on $n$ generators and 
$\FF_n$ is the free group on $n$ generators. 

Here is our first result tying certain finiteness 
properties of algebraic objects associated to a 
group $\pi$, under a filtered formality assumption. 

\begin{prop}
\label{prop:filtf}
Let $\pi$ be a finitely generated, filtered formal group, so that 
$\m (\pi) \cong \widehat{L}$, where $L=\L/J$ is a graded Lie algebra 
generated in degree $1$ and $J\subseteq \L^{\ge 2}$.  
If  $b_2(\M_1(\pi))< \infty$, then $\dim (J/[\L, J] )< \infty$.
\end{prop}

\begin{proof}
By the aforementioned duality between $\m (\pi)$ and $\M_1(\pi)$, 
the following holds: $b_k(\M_1(\pi))< \infty$ if and only 
if the inverse limit $\varprojlim_i H_k (\m/F_i)$ is a finite-dimensional 
$\k$-vector space, 
where $H_k(-)$ stands for Lie algebra homology, and $F=\{F_i\}_{i\ge 1}$ 
is the canonical inverse limit filtration on $\m=\m(\pi)$.

Our filtered formality assumption on the group 
$\pi$ yields Lie algebra isomorphisms 
\begin{equation}
\label{eq:mfi}
\m/F_i \cong \L/(\L^{\ge i} +J),
\end{equation}
for all $i\ge 1$. Moreover, the isomorphism $\widehat{L}\cong \m$ 
induces identifications $L/\Gamma_i(L) \cong \m/F_i$, which yield 
compatible Lie algebra maps $L\to \m/F_i$. Passing to homology, 
we obtain a natural homomorphism to the inverse limit, 
\begin{equation}
\label{eq:varphi}
\xymatrixcolsep{20pt}
\xymatrix{\varphi\colon H_{\bullet}(L) \ar[r]& \varprojlim_i H_{\bullet} (\m/F_i)}.
\end{equation}

Let us focus now on degree $\bullet=2$, where a well-known formula of H.~Hopf 
(see \cite{HS}) gives identifications  
\begin{equation}
\label{eq:hopf}
H_2(L)= J/[\L, J] \quad \text{and} \quad 
H_2 (\m/F_i)= (\L^{\ge i} +J)/(\L^{\ge i+1} +[\L,J]).
\end{equation}
Since all vector spaces in sight are graded and the map $\varphi$ 
respects degrees, we infer that the map $\varphi_2$ is injective. 
Putting things together completes the proof.
\end{proof}

\subsection{Hall--Reutenauer bases}
\label{subsec:grobner}

Once again, let $\L_n$ be the free Lie algebra on the 
(totally ordered) alphabet $\{ 1, \dots, n\}$, and let 
$\L_n''$ be its second derived Lie subalgebra. 
We analyze below a certain quotient of $\L_n''$; 
the {\em Hall--Reutenauer bases} constructed in 
\cite[Theorem 5.7, p.~112]{R} will be particularly 
well-adapted to our purposes.

\begin{prop}
\label{prop:freemeta}
For $n\ge 2$, the graded vector space $\L_n''/[\L_n, \L_n'']$ is infinite-dimensional.
\end{prop}

\begin{proof}
Given elements $h_1, \dots, h_k \in \L_n$, 
we denote by $[h_1, \dots, h_k] \in \L_n$ the iterated bracket
$[\dots [[h_1, h_2],  \dots, h_k]$. We set $\HH_0:= \{ 1, \dots, n\}$, 
and define inductively 
\begin{equation}
\label{eq:defhr}
\HH_i:= \{ [h_1, \dots, h_k] \}\, ,
\end{equation}
where $k \ge 2$, $h_1, \dots, h_k \in \HH_{i-1}$, and $h_1<h_2\ge \dots \ge h_k$. 
We pick a total order on the trees indexing the elements 
of $\HH_i$, compatible with degrees in $\L_n$ (that is, $h<h'$ if $\deg (h)<\deg(h')$), 
and decree that $h>h'$, for $h \in \HH_{i-1}$ and $h' \in \HH_{i}$.
As shown in \cite{R}, the Lie monomials from $\bigcup_{i\ge \ell} \HH_i$,  
form a basis for the vector space $\L_n^{(\ell)}$, for all $\ell$. 

Let $\L=\L_{\A}$ be the free Lie algebra on the alphabet $\A$, 
and $J\subseteq \L$ an ideal. A straightforward induction 
on degree establishes an equality of vector spaces, 
$[\L, J]=[\A, J]$. By a split surjectivity argument, it 
will be enough to check that
\begin{equation}
\label{eq:n2}
\dim \frac{\L''}{[\A, \L'']}= \infty \, ,
\end{equation}
where $\A= \{ x,y\}$ with $x<y$. Note that the Lie algebra $\L = \bigoplus_{i,j} \L_{i, j}$ 
is naturally bigraded, where $i$ is the $x$-degree and $j$ is the $y$-degree.

Clearly, $\HH_0= \{ x,y\}$. It follows from (\ref{eq:defhr}) that 
\begin{equation}
\label{eq:defh1}
\HH_1= \{ [xy^px^q]  \mid p\ge 1,\,  q\ge 0,\, p+q\ge 1  \}\, ,
\end{equation}
where $[xy^px^q]$ denotes $[x, \overbrace{y, \dots, y}^p, \overbrace{x, \dots, x}^q]$. 
In particular, $x$-$\deg (h)\ge 1$, for all $h\in \HH_1$.
Induction on $i\ge 1$ based on definition (\ref{eq:defhr}) shows that 
\begin{equation}
\label{eq:xdeg}
\text{$x$-$\deg (h) \ge 2^{i-1}$,\: for all $h\in \HH_i$} \, .
\end{equation}
In particular, all elements of the Hall--Reutenauer basis of $\L''$ 
have $x$-degree  $\ge 2$.

Consequently, if the vector space $\L''/[\A, \L'']$ is trivial in degree $i+2$, 
then necessarily $\ad_y$ induces a surjection $\L''_{2, i-1}\surj \L''_{2,i}$. Hence,
\begin{equation}
\label{eq:dims}
\dim \L''_{2,i} \le \dim \L''_{2, i-1}\, .
\end{equation}

Next, we compute dimensions in (\ref{eq:dims}). We infer 
from (\ref{eq:xdeg}), together with (\ref{eq:defhr}) 
and (\ref{eq:defh1}), that the following elements
form a basis of $\L''_{2,i}$: 
\begin{equation}
\label{eq:x2basis}
\{ [[xy^{p+1}], [xy^{q+1}]] \mid \text{$0\le p<q$ and $p+q=i -2$} \}\, .
\end{equation}
Therefore,
\begin{equation}
\label{eq:parity}
\dim \L''_{2,i}=
\begin{cases}
k & \text{if $i=2k+1$,}\\
k-1& \text{if $i=2k$.}
\end{cases}
\end{equation}

Consequently, the graded vector space $\L''/[\A, \L'']$ must 
be non-zero in all odd degrees $5$ and higher, since otherwise (\ref{eq:parity}) 
would contradict \eqref{eq:dims}. This verifies claim (\ref{eq:n2}), thereby 
completing our proof.
\end{proof}

\subsection{Proof of Theorem \ref{thm:meta}}
\label{subsec:pfmeta}

Let $\pi$ be a finitely generated group.  By assumption, there is an  
epimorphism $\varphi\colon \pi\surj \FF_n$, for some $n\ge 2$.  
Since the group $\FF_n$ is free, the map $\varphi$ admits a 
splitting. Hence, the induced homomorphism, 
$\bar\varphi\colon \pi/\pi'' \surj \FF_n/\FF_n''$, 
is again a split epimorphism.  By the homotopy functoriality 
of the $1$-minimal model construction, the map $\bar\varphi$ 
induces a $\cdga$ map, 
\begin{equation}
\label{eq:phistar}
\xymatrixcolsep{20pt}
\xymatrix{\bar\varphi^*\colon \M_1(\FF_n/\FF_n'') \ar[r]& \M_1(\pi/\pi'')},
\end{equation}
which is a split injection up to homotopy.

Suppose now that the conclusion of the theorem does not hold, i.e., 
suppose that the group $\pi/\pi''$ is  finitely presentable, or that it 
admits a $1$-finite $1$-model. It then follows from Corollary \ref{cor:infobs1} 
that $b_2(\M_1(\pi/\pi''))< \infty$.  Since, as we saw above, the map 
$\bar\varphi^*$ is split injective (up to homotopy), and since homology 
is a homotopy functor, we conclude that 
\begin{equation}
\label{eq:b2fn}
b_2(\M_1(\FF_n/\FF_n''))< \infty.
\end{equation}

Now, since the group $\FF_n/\FF_n''$ is filtered formal 
and (\ref{eq:b2fn}) holds, Proposition \ref{prop:filtf} implies that 
$\dim ( \L_n''/[\L_n, \L_n''] ) <\infty$. 
But this contradicts Proposition \ref{prop:freemeta}, and so we are done. 
\hfill $\Box$
 
\section{Cohomology jump loci, finiteness properties, and largeness}
\label{sect:cjl large}

\subsection{Cohomology jump loci}
\label{subsec:cjl}

Let $X$ be a path-connected space with fundamental group 
$\pi=\pi_1(X)$. The cohomology jump loci with coefficients in 
rank $1$ complex local systems on $X$ are powerful 
homotopy-type invariants of the space, that have been the 
subject of intense investigation in recent years, see 
for instance \cite{DPS-duke,PS-plms,DP-ccm,BW}.   
These loci sit inside the complex algebraic group 
$\widehat{\pi}:= \Hom(\pi, \C^{\times})$, 
and are defined for all $i,r\ge 0$ by
\begin{equation}
\label{eq:defjump}
\VV^i_r(X) =\{ \rho \in \widehat{\pi} \mid 
\dim H^i(X, \C_{\rho} )\ge r\}\, ,
\end{equation}
where $\C_{\rho}$ is the rank $1$ local system on $X$ associated to 
a representation $\rho\colon \pi\to \C^{\times}$, i.e., the vector 
space $\C$ viewed as a module over the group algebra $\C[\pi]$ 
via the action $g\cdot a=\rho(g) a$, for $g\in \pi$ and $a\in \C$. 
When the space 
$X$ is $q$-finite, the {\em characteristic varieties}\/ $\VV^i_r(X)$ 
are Zariski closed subsets of the character group $\widehat{\pi}$, 
for all $i\le q$ and $r\ge 0$, see \cite{PS-mrl}. It is easily seen that the 
the sets $\VV^i_r(X)$ with $i\le 1$ and $r\ge 0$ depend 
only on the group $\pi=\pi_1(X)$. 

Now let $\pi$ be a finitely generated group, and define 
$\VV^i_r(\pi):=\VV^i_r(K(\pi,1))$ for $i,r\ge 0$.  It is known that the sets 
$\VV^i_r(\pi)$ with $i\le 1$ and $r\ge 0$ depend only on 
the maximal metabelian quotient $\pi/\pi''$ (see e.g.~\cite{DPS-imrn,KS}); 
more precisely, the following equality holds,
\begin{equation}
\label{eq:jumpmeta}
\VV^i_r(\pi)=\VV^i_r(\pi/\pi'').
\end{equation}

The characteristic varieties have several useful naturality properties.  
For instance, suppose $\varphi\colon \pi\surj G$ is an epimorphism. 
Then the induced morphism on character groups, $\varphi^*\colon 
\widehat{G}\to \widehat{\pi}$, is injective and sends $\VV^1_r(G)$ 
into $\VV^1_r(\pi)$  for all  $r\ge 0$.  
Likewise, suppose that $H<\pi$ is a finite-index subgroup.  Then, 
as noted in \cite[Lemma 3.6]{PP}, the inclusion $\alpha\colon H\to \pi$ induces a 
morphism $\hat{\alpha}\colon \widehat{\pi}\to \widehat{H}$ with finite 
kernel, which sends $\VV^i_r(\pi)$ to $\VV^i_r(H)$ for all  $i,r\ge 0$.  

For the free groups $\FF_n$ of rank $n\ge 2$, we have that 
$\VV^1_r(\FF_n)=(\C^{\times})^n$ for $r\le n-1$ and $\VV^1_n(\FF_n)=\{1\}$. 
In general, though, the jump loci of a group can be arbitrarily complicated.  

\begin{example}
\label{ex:syz}
Let $f\in \Z[t_1^{\pm 1},\dots , t_n^{\pm 1}]$ 
be an integral Laurent polynomial with $f(1)=0$.   
Then, as shown in \cite{SYZ}, there is a finitely presented group $\pi$ 
with $\pi_{\ab}=\Z^n$ such that $\VV^1_1(\pi)$ coincides with 
the variety $\mathbf{V}(f):=\{t \in (\C^{\times})^n \mid f(t)=0\}$.  
\end{example}

\subsection{A finiteness obstruction}
\label{subsec:infvsBW}

Let $A$ be a connected $\cdga$.  Clearly, $H^1(A)=Z^1(A)$. 
For every $\omega \in H^1(A)$, the operator $d_{\omega}:= d+ \omega \, \cdot$ 
is a differential on $A^{\hdot}$. The {\em resonance varieties}\/ $\RR^i_r(A)$
are defined, for all $i,r\ge 0$, as the infinitesimal jump loci
\begin{equation}
\label{eq:defres}
\RR^i_r(A) =\{ \omega \in H^1(A) \mid 
\dim H^i(A, d_{\omega} )\ge r\}\, .
\end{equation}
When the $\cdga$ $A$ is $q$-finite, these sets are Zariski closed subsets of 
the affine space $H^1(A)$, for all $i\le q$ and $r\ge 0$, 
The following key result is a particular case of  \cite[Theorem B]{DP-ccm}. 

\begin{theorem}[\cite{DP-ccm}]
\label{thm:thmb}
Let $X$ be a $q$-finite space which admits a $q$-finite $q$-model $A$.
There is then a local analytic isomorphism between the germ at $1$ of 
$\pi_1(X)^{\what}$ and the germ at $0$ of $H^1(A)$, 
that identifies the germ at $1$ of $\VV^i_r(X)$ 
with the germ at $0$ of $\RR^i_r(A)$, for all $i\le q$ and $r\ge 0$. 
\end{theorem}

Recent work of Budur and Wang \cite{BW17}, building on the 
aforementioned theorem, provides a very strong 
finiteness obstruction for spaces, based on the geometry 
of their characteristic varieties. 

\begin{theorem} 
\label{thm:jumpobs}
If $X$ is $q$-finite and has a $q$-finite $q$-model, then each 
irreducible component of $\VV^i_r(X)$ passing through the 
identity $1\in \widehat{\pi}$ is an algebraic subtorus 
of the character group of $\pi=\pi_1(X)$, for all $i\le q$ and $r\ge 0$.
\end{theorem}

This theorem follows from  \cite[Corollary 2.2]{BW17} 
and Theorem \ref{thm:thmb}, and refines \cite[Theorem C(2)]{DP-ccm}; 
see also  \cite[Theorem 1.3(1)]{BW17}.  The Budur--Wang 
jump loci finiteness obstruction from 
Theorem \ref{thm:jumpobs} may be used to give a negative 
answer to Question \ref{mainpbm}, even for spaces $X$ 
which are finite CW-complexes. 

\begin{example}
\label{ex:fpresgrp}
Let $f$ be an integral Laurent polynomial in $n\ge 2$ variables, and 
assume its zero set in $(\C^{\times})^n$ contains the origin $1$, is irreducible 
but is not an algebraic subtorus; for instance, take $f(t)=\sum_{i=1}^n t_i -n$.  
Letting $\pi$ be a finitely presented group with $\VV^1_1(\pi)=\mathbf{V}(f)$ 
as in Example \ref{ex:syz}, we infer from Theorem \ref{thm:jumpobs} 
that the finite presentation complex of $\pi$ admits no $1$-finite $1$-model. 
\end{example}

Conversely, the existence of a $1$-finite $1$-model for a finitely generated 
group $\pi$ does not necessarily imply that $\pi$ is finitely presented.

\begin{example}
\label{ex:bbgrps}
Let $Y$ be a finite, connected CW-complex which is non-simply connected yet 
has $b_1(Y)=0$, and let $\pi$ be the Bestvina--Brady group associated to a flag 
triangulation of $Y$. It is proved in \cite[\S 10]{PS09} that $\pi$ is finitely generated
and $1$-formal, but not finitely presented. 
\end{example}

As the next family of examples illustrates, our infinitesimal finiteness 
obstruction from Theorem \ref{prop:infobsq} may be stronger than 
the Budur--Wang obstruction from Theorem \ref{thm:jumpobs}, 
even in the case when $q=1$. 

\begin{example}
\label{ex:subtler}
Consider the free metabelian group $\pi= \FF_n/\FF_n''$ with $n\ge 2$. 
The free group $\FF_n=\pi_1(\bigvee^n S^1)$ has a finite, formal 
classifying space; thus, Theorem \ref{thm:jumpobs} applies to $\FF_n$.
It follows that the varieties $\VV^i_r(\pi)\cong \VV^i_r(\FF_n)$ pass 
the Budur--Wang test for $i\le 1$ and $r\ge 0$. 

On the other hand, as we saw in the proof of Theorem \ref{thm:meta}, 
we have that $b_2(\M_1(\pi))= \infty$, 
and so the group $\pi$ admits no $1$-finite $1$-model. 
\end{example}

\subsection{Largeness}
\label{subsec:large}

A group $\pi$ is said to be {\em very large}\/ if it admits 
a free, non-cyclic quotient.  (Clearly, this property implies that $b_1(\pi)>0$.) 
Following Gromov \cite{Gr}, we say that $\pi$ is 
 {\em large}\/ if there is a finite-index subgroup $H<\pi$ 
such that $H$ surjects onto a free, non-cyclic group (i.e., $H$ is very large). 
In this definition, 
there is no loss of generality in assuming $H\triangleleft \pi$ is a normal 
subgroup.  In general, it is a difficult problem to decide whether a group is large. 
In \cite{Ko}, Koberda gives the following largeness test for groups 
admitting a finite presentation. 

\begin{theorem}[\cite{Ko}]
\label{thm:ko}
A finitely presented group $\pi$ is large if and only if there exists a 
finite-index subgroup $K<\pi$ such that $\VV^1_1(K)$ has infinitely 
many torsion points.  
\end{theorem}

As noted by Koberda and Suciu in \cite{KS}, the finite presentation 
assumption is crucial for this largeness criterion.  For instance, taking 
again $\pi=\FF_n/\FF_n''$ with $n\ge 2$, the variety 
$\VV^1_1(\pi)=(\C^{\times})^n$ has infinitely many torsion points, though 
the group $\pi$ is solvable, and thus not large.  

Also in \cite{KS}, the notion of RFR$p$ group (for $p$ a prime) is introduced.  
The class of groups which are RFR$p$ for all 
primes $p$ includes all groups of the form $\pi_1(C)$, for $C$ 
a smooth complex curve with $\chi(C)<0$, and all right-angled 
Artin groups. Using the criterion from Theorem \ref{thm:ko}, 
the following generalization of an old result of 
Baumslag and Strebel ensues.

\begin{theorem}[\cite{KS}]
\label{thm:ks}
Let $\pi$ be a finitely generated 
group which is non-abelian and RFR$p$ for infinitely 
many primes $p$. Then $\pi/\pi''$ is not finitely presented. 
\end{theorem}

Using a similar method, we may give an alternate proof 
of Theorem \ref{thm:meta}\ref{met1}.  

\begin{prop}
\label{prop:metabis}
Let $\pi$ be a finitely generated, very large group.  
Then $\pi/\pi''$ is not finitely presented. 
\end{prop}

\begin{proof}
By assumption, there is an epimorphism $\varphi\colon \pi\surj \FF_n$ to a free 
group of rank $n\ge 2$.  The induced morphism, $\hat\varphi\colon 
\widehat{\FF_n}\inj \widehat{\pi}$, embeds $\VV^1_1(\FF_n)=(\C^{\times})^n$ 
into $\VV^1_1(\pi)$.  Thus, this variety, which coincides with $\VV^1_1(\pi/\pi'')$, 
contains infinitely many torsion points.  If the group $\pi/\pi''$ were finitely presented, 
then, by Theorem \ref{thm:ko}, it would be large. But this is impossible, 
since $\pi/\pi''$ is solvable, and subgroups of solvable groups are again solvable.
\end{proof}

It is now natural to ask whether the above result holds 
in the more general setting of large groups.  More precisely:
Let $\pi$ be a finitely generated, large group; is it true that 
the maximal metabelian quotient $\pi/\pi''$ is not finitely presented?
As the next family of examples shows, the answer to this question is negative.

\begin{example}
\label{ex:largefp}
Let $\pi=\pi_1 \ast \pi_2$ be the free product of two non-trivial finite groups, 
of orders $m_1$ and $m_2$, with $m_1 m_2>4$. By \cite[Proposition 4, p.~6]{Serre80}, 
the kernel of the natural surjection, $\pi_1 \ast \pi_2 \surj \pi_1 \times \pi_2$, 
is a free group $\FF$ of rank $(m_1-1)(m_2-1)>1$. Hence, $\pi$ is a large group. 
On the other hand, $\pi'$ contains $\FF$ as a subgroup of finite index. Therefore, 
$\pi'$ is finitely generated; hence, the abelian group $\pi'/\pi''$ is finitely generated, 
as well. Since $\pi/\pi'$ is also a finitely generated abelian group, we infer that $\pi/\pi''$ 
is finitely presented.
\end{example}

\subsection{Large quasi-projective groups}
\label{subsec:qplarge}

Recall that a quasi-projective variety is a Zariski open subset of 
a projective variety.  We will say that a space $X$ is a 
{\em quasi-projective manifold}\/ if it is a connected, 
smooth, complex quasi-projective variety. Every such 
manifold has the homotopy type of a finite CW-complex.

We now turn to the question of deciding whether a quasi-projective 
group (i.e., a group that can be realized as the fundamental 
group of a quasi-projective manifold) is large.  It turns out that 
a complete answer to this question can be given in terms of 
``admissible" maps to curves.

A map $f\colon X \to C$ from a quasi-projective manifold $X$ to a 
smooth complex curve $C$ is said to be {\em admissible}\/ if it is 
regular, surjective, and has connected generic fiber. 
It is easy to see that the homomorphism on 
fundamental groups induced by such a map, 
$f_{\sharp}\colon \pi_1(X)\to \pi_1(C)$, is surjective. 
We denote by $\cE(X)$ the family of admissible 
maps to curves with negative Euler characteristic, 
modulo automorphisms of the target.  

Deep work of Arapura \cite{Ar} characterizes   
those positive-dimensional, irreducible components of the characteristic 
variety $\VV^1_1(X)$ which contain the origin of the character 
group $\pi_1(X)^{\what}$: all such components are connected, 
affine subtori, which arise by pullback of the character torus 
$\pi_1(C)^{\what}$ along the homomorphism 
$f_{\sharp}\colon \pi_1(X)\to \pi_1(C)$ induced 
by some map $f\in \cE(X)$. 

Suppose now that $C$ is a smooth complex curve with $\chi(C)<0$.  It is 
readily seen that the fundamental group $\pi=\pi_1(C)$ surjects onto 
a free, non-abelian group; consequently, $\pi$ is large. More 
generally, we have the following characterization of large, 
quasi-projective groups.

\begin{prop}
\label{prop:qproj}
Let $X$ be a quasi-projective manifold.  Then:
\begin{enumerate}
\item  \label{lg1}
$\pi_1(X)$ is large if and only if there is a finite cover $Y\to X$ such that 
$\mathcal{E}(Y)\ne \emptyset$. 
\item \label{lg2}
$\pi_1(X)$ is very large if and only if  
$\mathcal{E}(X)\ne \emptyset$. 
\end{enumerate}
\end{prop}

\begin{proof}
We only prove part \eqref{lg1}; the proof of part \eqref{lg2} is entirely similar.

Assume first that the group $\pi=\pi_1(X)$ is large.  There is then a 
finite-index subgroup $G<\pi$ and an epimorphism $\varphi\colon G\surj F$, 
where $F$ is a free group of rank $n\ge 2$.  
Let $Y\to X$ be the finite cover corresponding to $G$; as shown for instance 
in \cite[Lemma 4.1]{FS}, the space $Y$ is again a quasi-projective 
manifold.  Arguing as in the proof of \cite[Corollary 1.9]{Ar}, we see that 
the induced morphism on character groups, 
$\varphi^*\colon \widehat{F} \inj \widehat{G}$, 
takes $\VV^1_1(F)\cong (\C^{\times})^n$ to a positive-dimensional 
subvariety of $\VV^1_1(Y)$ passing through $1$.  The irreducible 
component of $\VV^1_1(Y)$ containing this subvariety, then, corresponds 
to an admissible map in $\mathcal{E}(Y)$.  

Conversely, suppose that there is a finite cover $Y\to X$ supporting 
an admissible map $f\colon Y\to C$, where $\chi(C)<0$.  
Composing the induced homomorphism $f_{\sharp}\colon 
\pi_1(Y)\surj \pi_1(C)$ with a surjection $\pi_1(C)\surj \FF_n$ 
($n\ge 2$), we obtain an epimorphism $\pi_1(Y)\surj \FF_n$, 
thereby showing that $\pi_1(X)$ is large.
\end{proof}

In view of \cite{Ar}, the above geometric test for very largeness 
has the following interpretation in terms of characteristic varieties. 
Assume without loss of generality  that $b_1(X)>0$; 
then $\mathcal{E}(X)\ne \emptyset$ if and only if the analytic germ 
at $1$ of $\VV^1_1(X)$ is not equal to $\{ 1\}$. 

\subsection{Resonance and largeness}
\label{subsec:reslarge}
To conclude this section, we rephrase the last condition in terms of resonance varieties.
As shown by Morgan \cite{Mo}, every quasi-projective manifold $X$ 
admits a finite model $A(\oX,D)$; such a `Gysin' model depends on 
a smooth compactification $\oX$ for which the complement 
$D=\oX\setminus X$ is a normal crossings divisor.  
Let $A$ be a Gysin model for $X$, or any one of the more general Orlik--Solomon 
models constructed by Dupont \cite{Dup}. In either case, let us note that all 
resonance varieties of $A$ have {\em positive weights}, i.e., they are invariant 
with respect to a $\C^{\times}$-action on $H^1(A)$ with positive weights.

\begin{prop}
\label{prop:qprojtest}
Let $X$ be a quasi-projective manifold with $b_1(X)>0$. Let $A$ be any 
Orlik--Solomon model of $X$.  Then  $\pi_1(X)$ is very large if and only if 
$\RR^1_1(A)\ne \{ 0\}$.
\end{prop}

\begin{proof}
As we already know, $\pi_1(X)$ is not very large if and only if the analytic 
germ at $1$ of $\VV^1_1(X)$ equals $\{ 1\}$, or equivalently, by Theorem \ref{thm:thmb}, 
the germ at $0$ of $\RR^1_1(A)$ equals $\{ 0\}$.   Since $\RR^1_1(A)$ 
has positive weights, all its irreducible components pass through $0$, 
and the aforementioned 
local equality is equivalent to the global equality $\RR^1_1(A)= \{ 0\}$.
\end{proof}

\begin{remark}
\label{rem:rescomp}
In practice, the geometric test for very largeness from Proposition \ref{prop:qproj} 
may be very difficult to implement, while the alternative test involving characteristic 
varieties requires an explicit presentation for the fundamental group.
On the other hand, once an Orlik--Solomon model $A$ has been constructed, 
the property that $\RR^1_1(A)= \{ 0\}$ can be verified concretely, using standard 
tools from computational commutative algebra.  The practical value of the 
infinitesimal test from Proposition \ref{prop:qprojtest} is 
well illustrated by the next class of examples.
\end{remark}

\begin{example}
\label{ex:partial}
Let $\Sigma_g$ be a compact, connected Riemann surface 
of genus $g$, and let $X=F_{\Gamma}(\Sigma_g)$ be the partial 
configuration space associated to a finite simple graph $\Gamma$.  
More concretely, if $n$ is the number of vertices of $\Gamma$, then 
$F_{\Gamma}(\Sigma_g)$ is the complement in 
$\Sigma_g^n$ of the union of the diagonals $z_i=z_j$, indexed by 
the edges of $\Gamma$. 

No presentation is available for the fundamental group 
$\pi_{\Gamma, g}=\pi_1(F_{\Gamma}(\Sigma_g))$, 
for arbitrary graph $\Gamma$ and genus $g$.  
On the other hand, as  noted in \cite{BMPP}, the 
quasi-projective manifold $F_{\Gamma}(\Sigma_g)$ admits an easily 
computable Orlik--Solomon model $A$.  Computing the resonance 
variety $\RR^1_1(A)$ leads to a complete, explicit description of 
$\mathcal{E}(F_{\Gamma}(\Sigma_g))$; such a description is given in 
\cite[Theorem 1.1]{BMPP}, for all $g\ge 0$ and for all finite graphs $\Gamma$. 
In particular, $\mathcal{E}(F_{\Gamma}(\Sigma_g))= \emptyset$, that is, 
$\pi_{\Gamma,g}$ is not very large, if and only if 
either $g=1$ and $\Gamma$ has no edges, or $g=0$ and  $\Gamma$ contains 
no complete subgraph on $4$ vertices. 
\end{example}

\section{A functorial $1$-minimal model map}
\label{sec:holmal}

We devote this section to the proof of Theorem \ref{thm:nat1model-intro}, and derive 
a topological interpretation.

\subsection{Holonomy Lie algebras}
\label{subsec:holo}

Given a $1$-finite $\cdga$ $A$, let $A_i=\Hom (A^i, \k)$ be the dual 
vector space.  Let $\mu^* \colon A_2 \to A_1\wedge A_1$ 
be the dual to the multiplication map 
$\mu \colon A^1\wedge A^1\to A^2$, and let $d^* \colon A_2\to A_1$ be 
the dual of the differential $d \colon A^1\to A^2$. We shall denote by 
$\L(A_1)$ the free Lie algebra on the $\k$-vector space $A_1$, 
and identify $\L_1(A_1)=A_1$ and  $\L_2(A_1)=A_1\wedge A_1$.

\begin{definition}[\cite{MPPS}]
\label{def:holo cdga}
The {\em holonomy Lie algebra}\/ of a $1$-finite $\cdga$ $A=(A^{\hdot},d)$ 
is the quotient of $\L(A_1)$ by the ideal generated by the image of the map 
$\partial_A=d^* + \mu^*$,
\begin{equation}
\label{eq:holo}
\h(A) = \L(A_1) / (\im (\partial_A)). 
\end{equation}
\end{definition}

Clearly, this construction is functorial.  Observe that 
the Lie algebra $\h(A)$ depends 
only on the $\cdga$ $A[1]$ constructed in Lemma \ref{lem:morefin}, 
i.e., on the sub-$\cdga$ 
\begin{equation}
\label{eq:a1}
\k \cdot 1 \oplus A^1 \oplus (d(A^1)+\mu(A^1\wedge A^1))
\end{equation}
of the truncation $A^{\le 2}$. In particular, $\h(A)$ is finitely presented. 

In the case when $d=0$, the above definition coincides with  
the classical holonomy Lie algebra of K.T.~Chen \cite{Ch}.  
In this situation, the Lie algebra $\h(A)$ inherits a natural 
grading from the free Lie algebra $\L(A_1)$, which is compatible 
with the Lie bracket.   Consequently, $\h(A)$ is a finitely-presented, 
graded Lie algebra, with generators in degree $1$ and relations in 
degree $2$.

In general, though, the ideal generated by $\im(\partial_A)$ 
is not homogeneous, and so the Lie algebra $\h(A)$ does not 
inherit a grading from $\L(A_1)$. Here is a simple example 
of a minimal, finite $\cdga$ for which this happens.

\begin{example}
\label{ex:heis}
Let $A=\bwedge(a_1,a_2,a_3)$ be the exterior algebra on generators $a_i$ 
in degree $1$, endowed with the differential given by $d{a_1}=d{a_2}=0$ 
and $d{a_3}=a_1 \wedge a_2$.  Identify $\L(A_1)$ with the free Lie algebra 
on dual generators $x_1,x_2,x_3$. Then the ideal $(\im(\partial_A))$ is 
generated by $x_3+[x_1,x_2]$, $[x_1,x_3]$, and $[x_2,x_3]$, and thus 
is not homogeneous. 
\end{example}

In the above example,  $\h(A)$ still admits the structure 
of a graded Lie algebra, with $x_1$ and $x_2$ in degree $1$, and $x_3$ in degree $2$.
Nevertheless, using a construction from \cite{SW}, we may define a minimal, finite 
$\cdga$ $A$ for which $\h(A)$ does not admit any grading compatible with the 
lower central series filtration.

\begin{example}
\label{ex:noncarnot}
Let $A=\bwedge(a_1,\dots, a_5)$, with differential given by 
$d{a_4}=a_1 \wedge a_3$, $d{a_5}=a_1 \wedge a_4+a_2 \wedge a_3$, 
and $d{a_i}=0$, otherwise. Then, as shown in \cite[Example 10.5]{SW}, 
$\h(A)$ is not isomorphic to $\gr(\h(A))$, the associated 
graded Lie algebra with respect to the LCS filtration.
\end{example}

\subsection{Holonomy and flat connections}
\label{subsec:holflat}

Next, we recall from \cite{DP-ccm} and \cite{MPPS} another 
(bifunctorial) construction which will be important to us. 
For a $\cdga$ $A$ and a Lie algebra $\g$, 
let $\F(A, \g)$ be the set of $\g$-valued {\em flat connections}\/ on $A$,    
i.e., the set of those elements $\omega \in A^1\otimes \g$ satisfying the
Maurer--Cartan equation,
\begin{equation}
\label{eq:mc}
d\omega + \tfrac{1}{2}[\omega, \omega]=0 \, .
\end{equation}

Suppose now that $A$ is $1$-finite. By \cite[Proposition 4.5]{MPPS}, 
the natural isomorphism $A^1\otimes \g \isom \Hom (A_1, \g)$ 
induces a natural identification,
\begin{equation}
\label{eq:flathol}
\xymatrix{\F(A, \g) \ar^(.38){\cong}[r] & \Hom_{\Lie} (\h (A), \g)} .
\end{equation}

Assume also that $\g$ is finite-dimensional.  We then let 
$\cC(\g)=\big(\bigwedge \g^*, d)$ be the Chevalley--Eilenberg 
complex of $\g$, that is, the $\cdga$ whose underlying graded 
algebra is the exterior algebra on $\g^*$, and whose differential 
 is the extension by the graded Leibnitz rule of the dual of 
the signed Lie bracket, $d=-\beta^*$, on the algebra generators, 
see e.g.~\cite{HS,FHT}.  There is then a natural isomorphism 
$A^1\otimes \g \isom \Hom (\g^*, A^1)$, which, 
by \cite[Lemma 3.4]{DP-ccm}, induces a natural identification,
\begin{equation}
\label{eq:flatcochains}
\xymatrix{\F(A, \g) \ar^(.36){\cong}[r] &\Hom_{\cdga} (\cC (\g), A)} .
\end{equation}

We will also consider the cochain functor $\wC$, defined on finitely 
generated Lie algebras and taking values in $1$-minimal $\cdga$s. 
This functor associates to a finitely generated Lie algebra $\h$ the 
inductive limit of $\cdga$s 
\begin{equation}
\label{eq:chat}
\wC (\h)= \varinjlim_n \cC (\h/\Gamma_n).
\end{equation}

The $1$-minimality 
property of $\wC (\h)$ is a consequence of the well-known 
fact that the functor $\cC$ sends finite-dimensional central 
Lie extensions to Hirsch extensions of $\cdga$s. For related 
constructions and results, we refer to \cite[Chapter 2]{FHT}.

\subsection{A classifying map}
\label{subsec:defclass}
As before, 
let $A$ be a $1$-finite $\cdga$, with holonomy Lie algebra $\h=\h(A)$. 
By (\ref{eq:flathol}), the identity map of $\h$ may be identified with 
the `canonical' flat connection, 
\begin{equation}
\label{eq:can}
\omega = \sum_i x_i^* \otimes x_i \in \F (A, \h(A)),
\end{equation}
where $\{ x_i \}$ is a basis for $A_1$ and $\{ x_i^* \}$ 
is the dual basis for $A^1$. This gives rise to a compatible 
family of flat connections, $\{ \omega_n \in \F(A, \h/\Gamma_n) \}_{n\ge 1}$.
Using the correspondence (\ref{eq:flatcochains}), we obtain a compatible 
family of $\cdga$ maps, $f_n \colon \cC (\h/\Gamma_n) \to A$. 
Passing to the limit, we arrive at a natural $\cdga$ map,
\begin{equation}
\label{eq:classmap}
\xymatrixcolsep{20pt}
\xymatrix{f\colon \wC (\h(A)) \ar[r]& A}.
\end{equation}

Our goal in this section is to show that this map is as an infinitesimal 
analog of the classifying map $X\to K(\pi_1(X), 1)$.  More precisely, 
we will prove the following theorem.  

\begin{theorem}
\label{thm:nat1model}
For any  $1$-finite $\cdga$ $A$, 
the classifying map $f\colon \wC (\h(A)) \to A$ 
is a $1$-minimal model map.
\end{theorem}

Although some form of this theorem may be known to specialists, 
we could not find it stated as such elsewhere.  
To prove the theorem, we need to show that 
$H^1(f)$ is an isomorphism and $H^2(f)$ is a 
monomorphism. It will be convenient to assume that 
$A$ is finite and $A^{>2}=0$. By naturality and the discussion 
from \S\ref{subsec:holo}, there is no loss of generality in doing so.

We will use the dual Chevalley--Eilenberg complex,   
which computes the untwisted Lie algebra homology 
$H_{\hdot}(\g)$ of an arbitrary Lie algebra $\g$.  This chain complex has the form 
$\{ \partial_n \colon \bwedge^n \g \rightarrow \bwedge^{n-1} \g \}_{n\ge 0}$, 
with differentials 
\begin{equation}
\label{eq:diff}
\partial_n (x_1\wedge \cdots \wedge x_n)= \sum_{i<j} (-1)^{i+j} 
\beta(x_i\wedge x_j) x_1\wedge \cdots \wedge \widehat{x_i}\wedge  \cdots \wedge 
\widehat{x_j}\wedge \cdots \wedge x_n\, .
\end{equation}

The chain complex $( \bwedge^\hdot \g, \partial )$ is functorial: 
given a morphism of Lie algebras, 
$\varphi \colon \g \to \h$,  the following diagram commutes, 
\begin{equation}
\label{eq:lienat}
\begin{gathered}
\xymatrixcolsep{28pt}
\xymatrix{
\bwedge^n \g   \ar^(.48){\partial_n}[r] \ar^{\bigwedge^n \varphi}[d] &
\bwedge^{n-1} \g \ar^{\bigwedge^{n-1} \varphi}[d]\\
\bwedge^n \h  \ar^(.48){\partial_n}[r]& \bwedge^{n-1} \h \, 
}
\end{gathered}
\end{equation}
for each $n\ge 1$.  Furthermore, both vertical arrows are surjective 
when $\varphi$ is an epimorphism.

We will come back to the proof of Theorem \ref{thm:nat1model} 
in \S\ref{subsec:pfclass}, after proving two technical lemmas.   

\subsection{Two stability properties}
\label{subsec:stab}
By definition, the generating vector space $V$ of a $1$-minimal $\cdga$ 
$(\bigwedge V, d)$ has an increasing, exhaustive filtration 
$\{ V^n\}_{n\ge 1}$ starting at $V^1=0$, with the property that each 
$(\bigwedge V^n, d)$ is a sub-$\cdga$ and each inclusion 
$(\bigwedge V^n, d)\inj (\bigwedge V^{n+1}, d)$ is a Hirsch extension.
We call such a filtration a {\em defining filtration}\/ for $(\bigwedge V, d)$. 
(Here and elsewhere we use upper indices for increasing filtrations and 
lower indices for decreasing filtrations.)
Defining filtrations having the following two stability properties 
will be important in our approach:

\begin{romenum}
\item \label{eq:stab1}
For all $m>n>1$, the natural inclusion $V^n \inj V^m$ induces 
an isomorphism
$H^1\big(\bigwedge V^n, d\big)\isom H^1\big(\bigwedge V^m, d\big)$.

\item \label{eq:stab2}
For all $m>n$, the kernel of the map 
$H^2(\bwedge V^n, d) \rightarrow H^2(\bwedge V^m, d)$ 
coincides with the kernel of the map 
$H^2(\bwedge V^n, d) \rightarrow H^2(\bwedge V^{n+1}, d)$.
\end{romenum}

Let $\h$ be a finitely generated Lie algebra. 

\begin{lemma}
\label{lem:cstab}
The defining filtration $V^n= (\h/\Gamma_n)^*$ on the $1$-minimal 
$\cdga$ $\wC (\h)$ satisfies properties (\ref{eq:stab1}) and (\ref{eq:stab2}).
\end{lemma}

\begin{proof}
We will translate the stability properties for the defining filtration 
$\{V^n\}_{n\ge 1}$ on $\wC (\h)$ in terms of 
Lie algebra homology, and use the commuting 
diagrams (\ref{eq:lienat}) in degrees $n \le 3$,
for the two canonical Lie projections,  
$p\colon \h/\Gamma_m \surj \h/\Gamma_{n+1}$ 
and $q\colon \h/\Gamma_{n+1} \surj \h/\Gamma_{n}$.

Property (\ref{eq:stab1}) is equivalent to the claim that the 
natural map $(\h/\Gamma_m)_{\ab} \to (\h/\Gamma_n)_{\ab}$ \
is an isomorphism, a claim which is obvious for $m>n>1$.  
Property (\ref{eq:stab2}) is equivalent to having 
an inclusion,
\begin{equation}
\label{eq:stab2dual}
\im \big(H_2(\h/\Gamma_{n+1}) \rightarrow H_2(\h/\Gamma_{n}) \big) 
\subseteq \im \big(H_2(\h/\Gamma_{m}) \rightarrow H_2(\h/\Gamma_{n})\big).
\end{equation}

We will verify this inclusion in a stronger form,  with $\h$ replacing 
$\h/\Gamma_{m}$ in (\ref{eq:stab2dual}). To this end, we start with an element 
$e_n\in \bwedge^2 (\h/\Gamma_{n})$ with $\partial_2 (e_n)=0$. If $[e_n]\in \im H_2(q)$,  
we may find $e_{n+1}\in \bwedge^2 (\h/\Gamma_{n+1})$ with $\partial_2 (e_{n+1})=0$ 
and $f_{n}\in \bwedge^3 (\h/\Gamma_{n})$ such that $\bwedge^2 (q)e_{n+1}= e_n + \partial_3 f_n$.

Now lift $f_n$ to  $f_{n+1}\in \bwedge^3 (\h/\Gamma_{n+1})$ via $\bigwedge^3 (q)$. 
Replacing  $e_{n+1}$ by $e_{n+1}- \partial_3 f_{n+1}$, we see that we may suppose that
$f_n=0$. Next, lift $e_{n+1}$ to $e \in \bigwedge^2 (\h)$ via $\bigwedge^2 (p)$. 
Since $\partial_2 (e_{n+1})=0$, we infer that $\partial_2 e \in \Gamma_{n+1}$.
By the definition of the $\lcs$ filtration \cite{Serre}, 
$\Gamma_{n+1}= \beta(\h \wedge \Gamma_n)$. This implies that 
$\partial_2 e= \partial_2 e'$, where $e' \in \h \wedge \Gamma_n$, 
and therefore $\bwedge^2 (q\circ p)(e')=0$. Therefore, $\partial_2 (e-e')=0$ and 
$\bwedge^2 (q \circ p)(e -e')=e_n$. Hence, $[e_n]\in \im H_2(q\circ p)$, 
which completes the proof.
\end{proof}

Next, we examine the requirement that 
\begin{equation}
\label{eq:stab2bis}
\ker \big(H^2(f_n) \colon H^2 \cC(\h/\Gamma_n) \rightarrow H^2A\big) \subseteq  
\ker \big(H^2(q) \colon H^2 \cC(\h/\Gamma_n) \rightarrow H^2 \cC(\h/\Gamma_{n+1})\big)\, ,
\end{equation}
where $H^2(q)$ is the map induced by the canonical Lie projection, 
$q\colon \h/\Gamma_{n+1} \surj \h/\Gamma_n$.
Recall that the $\cdga$ $A$ is supposed to be finite with $A^{>2}=0$. Dualizing, 
we infer that the existence of inclusion (\ref{eq:stab2bis}) is equivalent to the existence 
of the inclusion 
\begin{equation}
\label{eq:stab2dualbis}
\im \big(H_2(q)\colon H_2(\h/\Gamma_{n+1}) \rightarrow H_2(\h/\Gamma_n)\big) \subseteq
\im \big(H_2(f_n)\colon H_2 A \rightarrow H_2(\h/\Gamma_n)\big)\, .
\end{equation}

By construction, the dual of the restriction of the map 
$f_n \colon \cC (\h/\Gamma_n) \to A$ to the space of algebra 
generators of the source, $f_n^*\colon A_1 \rightarrow \h/\Gamma_n$, 
coincides with the composition $p_n \circ \iota$. Here, $\iota$ is the canonical inclusion 
$A_1 = \L^1 \inj \L$, where $\L$ denotes the free Lie algebra on $A_1$, and 
$p_n\colon \L \surj \h/\Gamma_n$ stands for the canonical Lie projection. 

The following commuting diagram describes 
the map between chain complexes in low degrees induced by the 
cochain map $f_n \colon \cC (\h/\Gamma_n) \rightarrow A$,
\begin{equation}
\label{eq:fndual}
\xymatrixcolsep{60pt}
\begin{gathered}
\xymatrix{
0   \ar^(.48){}[r] \ar^{}[d] & \bwedge^3(\h/\Gamma_n)\, \ar^{\partial_3}[d] \\
A_2  \ar^(.48){\bigwedge^2 (p_n \iota) \mu^*}[r]  \ar^{d^*}[d] 
& \bwedge^2(\h/\Gamma_n)\, \ar^{\partial_2}[d]\\
A_1  \ar^(.48){p_n \iota}[r] &  \h/\Gamma_n \, .
}
\end{gathered}
\end{equation}

We denote by $K\subseteq A_2$ the kernel of $d^*$. 
It follows from diagram (\ref{eq:fndual}) that the map 
$\bwedge^2 (p_n \iota) \mu^*$ 
sends $K$ into $\ker (\partial_2)$, and thus induces a map 
$[\bwedge^2 (p_n \iota) \mu^*] \colon K \rightarrow H_2(\h/\Gamma_n)$.
Again by duality, (\ref{eq:stab2dualbis}) is equivalent to 
the existence of the inclusion 
\begin{equation}
\label{eq:stab2dualfinal}
\im \big(H_2(q)\colon H_2(\h/\Gamma_{n+1}) \rightarrow H_2(\h/\Gamma_n)\big) \subseteq
\im \big([\bwedge^2 (p_n \iota) \mu^*] \colon K \rightarrow H_2(\h/\Gamma_n)\big)\, .
\end{equation}

Consider the map from Definition \ref{eq:holo}, 
$\partial= \partial_A \colon A_2 \rightarrow \L^1 \oplus \L^2$,
where we use the Lie bracket $\beta$ to identify $\bigwedge^2 \L^1$ 
with $\L^2$. Note that the map $p_n\colon \L \to \h/\Gamma_n$ is 
the composition of the canonical Lie projections  $\h \surj \h/\Gamma_n$ 
and $p\colon \L \surj \h$. 
We will repeatedly use the commuting diagrams (\ref{eq:lienat}) 
for $\varphi =p$, in degrees up to $3$. Consider the linear map 
\begin{equation}
\label{eq:delprime}
\partial' := \bwedge^2 (p)\circ \mu^*\colon K \rightarrow \bwedge^2 (\h).
\end{equation}
We first claim that the composition $\partial_2 \circ \partial'$
is the zero map. Indeed, for $y\in K$ we have that 
\begin{equation}
\label{eq:betacirc}
\beta \circ \bwedge^2(p) \mu^*y= p\circ \beta \mu^*y= p \circ \partial y=0.
\end{equation}

\begin{lemma}
\label{lem:partialprime}
If $A$ is a finite $\cdga$ with $A^{>2}=0$, then the induced map 
$[\partial']\colon K \rightarrow H_2(\h)$ is surjective.
\end{lemma}

\begin{proof}
We pick finite bases, $\{ x_i\}_{i\in I}$ for $A_1$ and $\{ y_{\lambda}\}_{\lambda \in \Lambda}$ 
for $A_2$.  By construction, $\h=\L/\rr$, where $\rr$ is the Lie ideal generated by 
$\{\partial y_{\lambda}\}_{\lambda \in \Lambda}$. 
For a length $q$ multi-index $J=(i_1, \dots, i_q)\in I^q$ and an element $r\in \L$, 
we abbreviate $\ad_{x_{i_1}}\circ \dots \circ \ad_{x_{i_q}}(r)\in \L$ by $\ad_J (r)$.
Clearly, $\rr$ is additively generated by the elements 
$\ad_J (\partial y_{\lambda})$, where $J\in I^q$ and $q\ge 0$.

We start with an element $e\in \bwedge^2 \h$ with the property 
that $\partial_2 e=0$, and pick a lift $f\in \bigwedge^2 \L$, via $\bwedge^2(p)$.
Since $\beta e=0$, we have that $\beta f\in \rr \cap \L^{\ge 2}$. Write 
$\beta f= \sum c_{J, \lambda} \ad_J (\partial y_{\lambda})$. 
Since $\beta f\in \L^{\ge 2}$, we must have
$\sum_{\lambda \in \Lambda} c_{\emptyset, \lambda} d^*y_{\lambda}=0$. 
Hence, the element $y= \sum_{\lambda \in \Lambda} c_{\emptyset, \lambda} y_{\lambda}$ 
belongs to $K$.

If $J$ has positive length, we set $\widehat{J}_1:= (i_2, \dots, i_q)$. 
Clearly, the element 
\begin{equation}
\label{eq:f'def}
f':= \sum_{q>0, \lambda\in \Lambda} c_{J, \lambda} x_{i_1}\wedge 
\ad_{\widehat{J}_1} (\partial y_{\lambda})
\end{equation} 
belongs to $\L^1\wedge \rr$, and 
\begin{equation}
\label{eq:betaf'}
\beta f'=  \sum_{\substack{q>0\\ \lambda\in \Lambda}} 
c_{J, \lambda} \ad_J (\partial y_{\lambda})\, . 
\end{equation}

On the other hand, $\beta f= \beta f'+ \beta \mu^* y$, since $y\in K$. 
Hence, $\partial_2 (f-f'-\mu^*y)=0$, and $\bwedge^2(p)(f-f'-\mu^*y)= 
e-\partial' y$, by construction. 

Using the fact that the free Lie algebra $\L$ has vanishing homology 
in degrees greater than $2$ \cite{HS}, we may find an element 
$v \in \bwedge^3 \L$ such that
$f-f'-\mu^*y= \partial_3 v$. We infer that $e-\partial' y= \partial_3 u$, 
where $u=\bwedge^3(p)v$. Hence, $[e]=[\partial'y]$, as asserted.
\end{proof}

\subsection{Proof of Theorem \ref{thm:nat1model}}
\label{subsec:pfclass}
Once again, let $f\colon \wC (\h(A)) \to A$ be the $\cdga$ 
map from (\ref{eq:classmap}). 
We need to show that $H^1(f)$ is an isomorphism and $H^2(f)$ is 
a monomorphism. In both cases, we will use the fact that 
\begin{equation}
\label{eq:hic}
H^i\big(\wC (\h)\big)=\varinjlim_n H^i \left(\cC (\h/\Gamma_n)\right).
\end{equation}

For $i=1$, by Lemma \ref{lem:cstab}(\ref{eq:stab1}) 
and duality, we need to verify that $H_1(f_2)$ is an isomorphism. 
For $n=2$, note that in diagram (\ref{eq:fndual}) 
$\partial_2=0$ and $p_2 \iota$ is surjective. 
Hence, we are left with checking that $H_1(f_2)$ is an injection. 
So, we assume that $p_2 \iota y=0$, i.e., $y\in A_1$ belongs to $\rr$ 
modulo $\Gamma_2$ in the free Lie algebra $\L$. It follows that 
necessarily $y\in \im (d^*)$, and we are done. 

For $i=2$, note that property (\ref{eq:stab2bis}) readily implies 
the injectivity of $H^2(f)$. The discussion from the preceding 
subsection reduces our proof to checking property (\ref{eq:stab2dualfinal}). 
To verify this property, pick an arbitrary element $[e_n]\in \im H_2(q)$.
 By property (\ref{eq:stab2dual}) in strong form,
$[e_n]=H_2(p'_n)[e]$ for some $[e]\in H_2(\h)$, where 
$p'_n\colon \h \surj \h/\Gamma_n$ denotes the canonical 
Lie projection. By Lemma \ref{lem:partialprime},
$[e]\in \im [\partial']$. Putting things together, we 
conclude that $[e_n]$ belongs to the image of the map 
$[\bwedge^2 (p_n \iota) \mu^*] \colon 
K \rightarrow H_2(\h/\Gamma_n)$.  This completes the proof.
\hfill $\Box$

\subsection{Topological interpretation}
\label{subsec:malformula}

As noted in \cite[p.~49]{FHT}, 
a $1$-minimal $\cdga$ $(\bigwedge V, d)$ admits a {\em canonical}\/ 
defining filtration, $\{ W^n \}_{n\ge 1}$, inductively defined by 
$W^1=0$ and 
\begin{equation}
\label{eq:wfilt}
W^{n+1}=d_{| V}^{-1} (\bwedge^2 W^n)\, .
\end{equation}
It is easy to check by induction that the inclusion $V^n \subseteq W^n$  
holds for all $n\ge 1$, and for any defining filtration $\{ V^n \}_{n\ge 1}$ 
on $(\bigwedge V, d)$.

\begin{lemma}
\label{lem:canfilt}
A defining filtration for the $1$-minimal $\cdga$ 
$(\bigwedge V, d)$ is canonical if and only if it 
has the stability properties (\ref{eq:stab1}) and (\ref{eq:stab2}).
\end{lemma}

\begin{proof}
First observe that $H^1(\bigwedge V^n)= V^n \cap W^2$, for all $n\ge 1$.

Assume now that $V^n=W^n$, for all $n$. Then clearly 
$V^n \cap W^2= V^m \cap W^2= W^2$ for all $m>n>1$, which proves (\ref{eq:stab1}).
To check (\ref{eq:stab2}), let us start with an element $\alpha \in \bwedge^2 W^n$ 
with the property that $\alpha =dv_m$ with $v_m \in W^m$. Clearly,
$v_m \in W^{n+1}$, and we are done.

Conversely, let us assume that properties (\ref{eq:stab1}) and (\ref{eq:stab2}) hold, 
and let us prove by induction that $W^n \subseteq V^n$, for all $n\ge 2$. By
(\ref{eq:stab1}), $V^n \cap W^2= V^2 \cap W^2$ for all $n\ge 2$, 
which shows that $W^2\subseteq V^2$. For the induction step, start with an
arbitrary element $v_{n+1} \in W^{n+1}$. Then $v_{n+1} \in V^m$ 
for some $m>n$. We infer that $dv_{n+1} \in \bwedge^2 V^n$, by the construction
of the canonical filtration and the inductive hypothesis.
Since the cohomology class $[dv_{n+1}]\in H^2 (\bigwedge V^n)$ 
goes to zero in $H^2 (\bigwedge V^m)$, property (\ref{eq:stab2}) implies that
$dv_{n+1}=du_{n+1}$, for some $u_{n+1}\in V^{n+1}$. Therefore, 
$v_{n+1}- u_{n+1}\in W^2 \subseteq V^{n+1}$. Hence, $v_{n+1}\in V^{n+1}$,
and we are done.
\end{proof}

As an application, we recover, in a self-contained way, a theorem of 
Berceanu, M\u{a}cinic, Papadima, and Popescu \cite[Theorem 3.1]{BMPP};  
in turn, that theorem generalizes a result of Bezrukavnikov 
\cite[Lemma 3.1 and Proposition 4.0]{Bez}, which is neatly 
summarized  in \cite[\S 4.3]{BH}.

\begin{corollary}
\label{cor:malholo}
If $A$ is a $1$-finite $1$-model of $\pi$, then the Malcev Lie algebra 
$\m (\pi)$ is isomorphic to the $\lcs$ completion of the holonomy Lie 
algebra $\h (A)$.
\end{corollary}

\begin{proof}
We start by recalling Sullivan's duality result \cite{Su77} for 
Malcev Lie algebras and $1$-minimal models of finitely 
generated groups (see also \cite[\S 6]{DP-ccm} 
and \cite[\S 6]{SW}).
By dualizing the canonical filtration of $\M_1(\pi)$, we 
obtain a tower of central extensions of finite-dimensional 
nilpotent Lie algebras, 
\begin{equation}
\label{eq:maltower}
\xymatrixcolsep{18pt}
\xymatrix{\cdots \ar@{->>}[r] & \m_{n+1} \ar@{->>}[r] & \m_n 
\ar@{->>}[r] & \cdots \ar@{->>}[r] & \m_1 =\{0\}}. 
\end{equation}

The Malcev Lie algebra $\m (\pi)$ is isomorphic to the inverse limit of the 
tower (\ref{eq:maltower}), endowed with the inverse limit filtration.
Our assumptions imply that the group $\pi$ and the $\cdga$ 
$A$ have the same $1$-minimal model. By Theorem \ref{thm:nat1model} 
and Lemmas \ref{lem:cstab} and \ref{lem:canfilt},
the above tower of Lie algebras is isomorphic to the tower
\begin{equation}
\label{eq:holtower}
\xymatrixcolsep{18pt}
\xymatrix{\cdots \ar@{->>}[r] & \h (A)/\Gamma_{n+1} \ar@{->>}[r] 
& \h (A)/\Gamma_{n} 
\ar@{->>}[r] & \cdots \ar@{->>}[r] &  \h (A)/\Gamma_{1} =\{0\}}, 
\end{equation}
whose inverse limit is  the $\lcs$ completion of $\h (A)$. 
This completes the proof.
\end{proof}

\section{A complete finiteness obstruction for finitely generated groups}
\label{sec:complobstr}

For $q=1$, the finiteness property from Question \ref{mainpbm} 
depends only on the finitely generated group $\pi=\pi_1(X)$. More 
precisely, it depends only on its $1$-minimal model $\M_1(\pi)$, or 
equivalently, on the Malcev Lie algebra $\m (\pi)$. In this situation, 
we provide a complete answer to Question \ref{mainpbm}, as follows.

\begin{theorem}
\label{thm:complobstr}
A space with finitely generated fundamental group $\pi$ admits a 
$1$-finite $1$-model if and only if the Malcev Lie algebra $\m (\pi)$
is the $\lcs$ completion of a finitely presented Lie algebra.
\end{theorem}

The rest of this section will be devoted to a proof of this theorem. 
We start with some preparatory notation and a lemma.
Let $\L= \L (x_i)_{i\in I}$ be a finitely generated free Lie algebra. 
For each $k\ge 1$ and each $k$-tuple $J=(i_1, \dots, i_k)\in I^k$, 
set $\abs{J}=k$, and define
\begin{equation}
\label{eq:beta}
\beta(J):=\ad_{x_{i_1}}\circ \dots \circ \ad_{x_{i_{k-1}}}(x_{i_k})\in \L^k.
\end{equation}
An easy induction shows that each graded piece $\L^{k}$ is generated 
by the Lie monomials $\beta(J)$ with $\abs{J}=k$.

\begin{lemma}
\label{lem:liepres}
Let $L=\L/\rr$ be a finitely presented Lie algebra. Then 
$L$ admits an equivalent finite presentation $\cP$, having 
only ``linear plus quadratic" relations.
\end{lemma}

\begin{proof}
The fact that $L$ is finitely presented allows us to find an integer $N\ge 2$ 
and constants $c^J_{\lambda}$, where $1\le \abs{J}\le N$ and 
$\lambda \in \Lambda$ with $\Lambda$ finite, with the property 
that the Lie ideal $\rr$ is generated by the elements 
$r_{\lambda}:=\sum_{J} c^J_{\lambda}\, \beta(J)$.

For each $n\ge 2$, let us define a finite Lie algebra presentation 
$\cP_n$ as having generators $y_J$ indexed by $1\le \abs{J}\le n$, 
and relators $y_J- [y_{i_1}, y_{\widehat{J}_1}]$,
indexed by $2\le \abs{J}\le n$, where $\widehat{J}_1:= (i_2, \dots, i_k)$. 
Furthermore, let us define a finite Lie algebra presentation $\cP$ 
by adding to the relators from $\cP_N$ the elements 
$\rho_{\lambda}:=\sum_{1\le \abs{J}\le N} c^J_{\lambda}\, y_J$ 
for $\lambda \in \Lambda$.

The Lie presentations $\cP_n$ and  $\cP_{n-1}$ are related by 
the morphisms of presentations $\phi \colon \cP_{n-1} \to \cP_{n}$ 
and $\psi \colon \cP_{n} \to \cP_{n-1}$.
The map $\phi$ sends the generator $y_J$ to $y_J$ for each $J$, 
while $\psi$ sends the generator $y_J$ to $y_J$ for each $J$ with 
$\abs{J}<n$, and maps $y_J$ to $[y_{i_1}, y_{\widehat{J}_1}]$ 
if $\abs{J}=n$. It is straightforward to check that the associated 
Lie algebra morphisms are inverse to each other.

We infer that the Lie algebra morphism $\kappa \colon \L \rightarrow L_N$ 
which sends $x_i$ to $y_i$ is an isomorphism, where $L_N$ denotes the Lie algebra 
associated to $\cP_N$. On the other hand, a straightforward induction shows that 
$\kappa$ sends $\beta(J)$ to $y_J$, for all $J$ with $1\le \abs{J}\le N$.
Finally, for each $\lambda \in \Lambda$ the relator $r_{\lambda}$ is 
identified by $\kappa$ with the element $\rho_{\lambda}$.
This completes the proof.
\end{proof}

{\bf Proof of Theorem \ref{thm:complobstr}.} 
For the forward implication, 
let $\pi=\pi_1(X)$ be a finitely generated group; by the discussion 
from \S\ref{subsec:algmod}, we may assume that $\pi$ 
has a $1$-finite $1$-model $A$. Furthermore, by Lemma \ref{lem:morefin}, 
we may assume that $A$ is finite with $A^{>2}=0$. Denote by 
$\h=\h(A)$ the holonomy Lie algebra of $A$. Clearly, the Lie algebra 
$\h$ is finitely presentable. By Corollary \ref{cor:malholo}, 
the Malcev Lie algebra $\m=\m (\pi)$ is the $\lcs$ completion 
of $\h$, and we are done.

Conversely, suppose that $\m (\pi)$ is the $\lcs$ completion of 
a finitely presented Lie algebra $L$ as above. Note that all relators 
of the presentation $\cP$ from Lemma \ref{lem:liepres} are elements 
of $\L^{\le 2}$, by construction. By dualizing this presentation, 
we obtain a finite $\cdga$ $A$ with $A^{>2}=0$, whose differential 
$d\colon A^1 \to A^2$ is dual to the degree one part of the relator 
map, and with multiplication $\mu\colon A^1 \wedge A^1 \to A^2$ dual 
the degree two part of the relator map.  By construction, the holonomy 
Lie algebra $\h=\h(A)$ is isomorphic to $L$.  Moreover, 
Theorem \ref{thm:nat1model} implies that $\wC (\h) \simeq_1 A$.

On the other hand, $\Omega (X) \simeq_1 \M_1(\pi)$, as noted before. 
It is therefore enough to check that the $\cdga$s $\wC (\h)$ and 
$\M_1(\pi)$ are isomorphic.  Exploiting the fact that both $\cdga$s 
are inductive limits of $\cdga$ towers of Hirsch extensions, 
this may be seen as follows. 
The dual of the tower for  $\wC (\h)$ is by construction the Lie tower 
$\{ \h/\Gamma_{n+1} \rightarrow \h/\Gamma_{n} \}_{n\ge 1}$.  
Furthermore, our assumption that $\m(\pi)\cong \widehat{L}$ 
implies that the dual of the tower for $\M_1(\pi)$ is 
isomorphic to $\{ L/\Gamma_{n+1} \rightarrow L/\Gamma_{n} \}_{n\ge 1}$. 
Hence, $\wC (\h)\cong \M_1(\pi)$, thereby establishing our claim, 
and completing the proof.
\hfill $\Box$

\begin{ack}
We are grateful to Barbu Berceanu for help with Proposition \ref{prop:freemeta}.
This work was initiated at the Centro di Ricerca Matematica 
Ennio De Giorgi in Pisa, Italy, in February 2015. The authors wish to 
thank the organizers of the Intensive Research Period on 
Algebraic Topology, Geometric and 
Combinatorial Group Theory for their warm hospitality, and 
for providing an inspiring mathematical environment. 
\end{ack}

\newcommand{\arxiv}[1]
{\texttt{\href{http://arxiv.org/abs/#1}{arxiv:#1}}}
\newcommand{\arx}[1]
{\texttt{\href{http://arxiv.org/abs/#1}{arXiv:}}
\texttt{\href{http://arxiv.org/abs/#1}{#1}}}
\newcommand{\doi}[1]
{\texttt{\href{http://dx.doi.org/#1}{doi:#1}}}
\renewcommand{\MR}[1]
{\href{http://www.ams.org/mathscinet-getitem?mr=#1}{MR#1}}

\end{document}